\documentclass[11pt,reqno]{amsart}

\usepackage[l2tabu,orthodox]{nag}

\usepackage{amsaddr}
\usepackage{mathrsfs}

\usepackage[all]{xy}

\usepackage{hyperref}

\usepackage{marginnote}
\usepackage{enumerate}

\usepackage[swedish,english]{babel}
\usepackage[T1]{fontenc}
\usepackage[latin1]{inputenc}

\usepackage{a4wide}

\usepackage{amsmath,amssymb,amsthm}

\theoremstyle{plain}
\newtheorem{theorem}{Theorem}
\newtheorem{lemma}[theorem]{Lemma}
\newtheorem{prop}[theorem]{Proposition}
\newtheorem{cor}[theorem]{Corollary}
\theoremstyle{definition}
\newtheorem{defn}[theorem]{Definition}

\newtheorem{exmp}[theorem]{Example}

\newtheorem{remark}[theorem]{Remark}

\newcommand{\N}{\mathbb{N}}
\newcommand{\Z}{\mathbb{Z}}

\newcommand{\Ng}{\mathcal{N}_g}
\newcommand{\NOne}{\mathcal{N}_1}

\DeclareMathOperator{\spanK}{span}

\DeclareMathOperator{\degree}{deg}

\title{Group gradations on Leavitt path algebras}

\date{\today}

\begin{document}

\author{Patrik Nystedt}
\address{Department of Engineering Science,
University West,
SE-46186 Trollh\"{a}ttan, Sweden}
\email{patrik.nystedt@hv.se}

\author{Johan \"{O}inert}
\address{Department of Mathematics and Natural Sciences,
Blekinge Institute of Technology,
SE-37179 Karlskrona, Sweden}
\email{johan.oinert@bth.se}


\begin{abstract}
Given a directed graph $E$ and an associative unital ring $R$ one may define
the Leavitt path algebra with coefficients in $R$, denoted by $L_R(E)$.
For an arbitrary group $G$,
$L_R(E)$ can be viewed as a $G$-graded ring.
In this article, we show that
$L_R(E)$ is always nearly epsilon-strongly $G$-graded.
We also show that if $E$ is finite, then $L_R(E)$
is epsilon-strongly $G$-graded.
We present a new proof
of Hazrat's characterization of strongly $\Z$-graded Leavitt path algebras,
when $E$ is finite.
Moreover, if $E$ is row-finite and has no source,
then we show that $L_R(E)$ is strongly $\Z$-graded if and only if $E$ has no sink.
We also use a result concerning Frobenius epsilon-strongly $G$-graded rings,
where $G$ is finite, to obtain criteria which ensure that $L_R(E)$
is Frobenius over its identity component.
\end{abstract}

\maketitle


\section{Introduction}\label{sec:Intro}

A Leavitt path algebra associates with a directed graph $E$ 
and an associative unital ring $R$, an $R$-algebra $L_R(E)$.
These algebras are algebraic analogues of graph $C^*$-algebras
and are also natural generalizations of Leavitt algebras of type $(1,n)$
constructed in \cite{leavitt1962}.
Leavitt path algebras were introduced, in the case when $R$ is a field,
by Ara, Moreno and Pardo \cite{AMP2007},
and almost simultaneously by Abrams and Aranda Pino \cite{AAP2005}, 
using a different approach.
In the case when $R$ is a commutative ring these structures
were defined by Tomforde in \cite{T2011}.
Throughout this article we will follow Hazrat \cite{H2013A} and consider Leavitt path algebras
with coefficients
in an arbitrary associative unital ring $R$
(see Definition~\ref{def:LPA}).
Each Leavitt path algebra may, in a canonical way, be viewed as a $\Z$-graded ring (see Section~\ref{sec:leavittpathalgebras}).
The characterization of Leavitt path algebras (e.g. simple, purely infinite simple,
semisimple, prime, finite dimensional, locally finite, exchange, artinian, noetherian) 
in terms of properties of the underlying graph has been the subject 
of many studies, see Abrams' extensive survey \cite{AbramsDecade}, 
and the references therein.
Many of these investigations have, however, been carried out
without taking into account the $\mathbb{Z}$-graded structure of $L_R(E)$.
Although there are previous studies (see e.g. Ara, Moreno and Pardo \cite{AMP2007}, Tomforde \cite{T2007})
where the $\Z$-graded structure of $L_R(E)$ has been taken into account,
it seems to the authors of the present article that Hazrat \cite{H2013A}
was first to utilize this philosophy systematically, and to e.g. study properties of the $\Z$-gradation itself.
In loc. cit. he proves, among many other things, the following result.

\begin{theorem}[Hazrat \cite{H2013A}]\label{thm:Hazrat2}
If $E$ is a finite directed graph and $R$ is an associative unital ring,
then the Leavitt path algebra $L_R(E)$ 
is strongly $\Z$-graded if and only if $E$ has no sink.
\end{theorem}

To motivate the approach taken in this article, 
consider the following directed graphs.
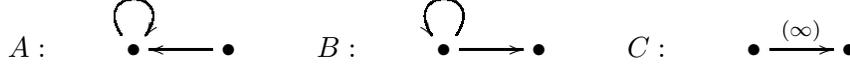
\begin{displaymath}
	\xymatrix{
A:        & \bullet \ar@(ul,ur)        & \ar[l] \bullet & 
B:        & \bullet \ar@(ur,ul) \ar[r] & \bullet        & 
C:        & \bullet \ar[r]^{(\infty)}  & \bullet \\
	}	
\end{displaymath}
According to Theorem~\ref{thm:Hazrat2}, the graph $A$ produces a
strongly $\Z$-graded Leavitt path algebra, but the graphs $B$ and $C$ do not.
In an attempt to single out a class of well-behaved group graded rings which
naturally includes all crossed products by unital twisted partial actions, 
Nystedt, \"{O}inert and Pinedo \cite{NOP2016} recently introduced a 
generalization of unital strongly graded rings called \emph{epsilon-strongly graded rings}
(see Definition~\ref{def:epsilongraded}).
In this article, we establish the following result which shows that
$L_R(B)$ is epsilon-strongly $\Z$-graded.

\begin{theorem}\label{maintheoremfinite}
If $E$ is a finite directed graph and $R$ is an associative unital ring,
then the Leavitt path algebra $L_R(E)$ is epsilon-strongly $\Z$-graded.
\end{theorem}

It turns out that $L_R(C)$ is not epsilon-strongly $\Z$-graded (see Example~\ref{exmp:notepsilongraded}).
To capture Leavitt path algebras corresponding to the graph $C$
we introduce the class of 
\emph{nearly epsilon-strongly graded rings} (see Definition~\ref{def:nearlyepsilon}),
which is a further weakening of the notion of a strongly graded ring,
and show the following result.

\begin{theorem}\label{maintheorem}
If $E$ is a (possibly non-finite) directed graph and $R$ is an associative unital ring,
then the Leavitt path algebra $L_R(E)$ is nearly epsilon-strongly $\Z$-graded.
\end{theorem}

Here is a detailed outline of this article.

In Section~\ref{sec:sunitalmodules}, we recall some notions concerning
unital and $s$-unital modules that we need in the sequel.

In Section~\ref{sec:gradedrings}, we provide the necessary background on 
group graded rings for use in subsequent sections. We define the following notions:
symmetrically graded ring,
epsilon-strongly graded ring, nearly epsilon-strongly graded ring,
non-degenerately graded ring and
strongly non-degenerately graded ring
(see Definitions~\ref{def:symmetricallygraded}, \ref{def:epsilongraded}, \ref{def:nearlyepsilon},
\ref{def:nondegenerate} resp. \ref{def:weaklynondegenerate})
and we show a result connecting these notions (see Proposition~\ref{epsilonweakly}).

In Section~\ref{sec:leavittpathalgebras}, we accomodate the Leavitt path algebra 
framework that we need in the sequel; in particular standard $G$-gradations
on Leavitt path algebras, for an arbitrary group $G$. 
In the same section, we prove $G$-graded versions of Theorem~\ref{maintheoremfinite}
and Theorem~\ref{maintheorem} (see Theorem~\ref{finitegraphepsilonstrongly}
and Theorem~\ref{graphepsilonstrongly}) through a series of propositions.
At the end of the section, we apply a result, by Pinedo and the authors of the present article,
concerning epsilon-strongly $G$-graded rings, to obtain criteria
which ensure that a Leavitt path algebra is Frobenius over its identity component,
when the group $G$ is finite (see Theorem~\ref{maintheoremfrobenius}).

In Section~\ref{sec:Zgraded}, we provide a useful characterization of strongly $\Z$-graded rings (see Proposition~\ref{prop:stronglyZgraded}) and several corollaries.

In Section~\ref{sec:ZgradedLPA}, we obtain a new proof of Hazrat's Theorem~\ref{thm:Hazrat2}
(see Proposition~\ref{prop:Hazrat}) by using our results from Section~\ref{sec:leavittpathalgebras}
and Section~\ref{sec:Zgraded}.
We also characterize strongly $\Z$-graded Leavitt path algebras
over row-finite graphs without sources (see Theorem~\ref{thm:rowfinite}).

\section{s-unital modules}\label{sec:sunitalmodules}

In this section, we recall some notions concerning
unital and $s$-unital modules that we need in the sequel.
Throughout this article all rings are associative
but not necessarily unital.
For the rest of this section $T$ and $T'$ denote rings.
Recall the following definitions.

If $M$ is a left $T$-module (right $T'$-module),
then $M$ is called \emph{left (right) unital} if there is
$t \in T$ ($t' \in T'$) such that, for all $m \in M$, the relation 
$t m = m$ ($m t' = m$) holds. 
In that case $t$ ($t'$) is said to be \emph{a left (right) identity for $M$}.
If $M$ is a $T$--$T'$-bimodule, then $M$ is called \emph{unital}
if it is unital both as a left $T$-module and as a right $T'$-module.
The ring $T$ is said to be \emph{left (right) unital} if it is
left (right) unital as a left (right) module over itself.
The ring $T$ is called \emph{unital} if it is unital
as a bimodule over itself.

In this article, we will use the notion of s-unitality,
introduced in \cite{tominaga1976}.
Namely, if $M$ is a left $T$-module (right $T'$-module), then
$M$ is said to be \emph{$s$-unital} if for each $m \in M$ the relation
$m \in Tm$ ($m \in mT'$) holds. 
If $M$ is a $T$--$T'$-bimodule, then $M$ is said to be
\emph{$s$-unital} if it is $s$-unital both as a left $T$-module
and as a right $T'$-module.
The ring $T$ is said to be \emph{left (right) $s$-unital}
if it is left (right) $s$-unital as a left (right) module over itself.
The ring $T$ is said to be \emph{$s$-unital} if it is $s$-unital
as a bimodule over itself.

\begin{remark}\label{tominaga}
From \cite[Theorem 1]{tominaga1976} it follows that if $M$ is a left $T$-module (right $T'$-module),
then $M$ is s-unital if and only if for all $n \in \mathbb{N}$
and all $m_1,\ldots,m_n \in M$ there exists $t \in T$ ($t' \in T$) such that,
for each $i \in \{ 1 , \ldots , n \}$, the equality
$t m_i = m_i$ ($m_i t' = m_i$) holds.
\end{remark}

Let $M$ be a left $T$-module (right $T'$-module). 
Recall that $M$ is called
\emph{torsion-free} if for all non-zero $m \in M$ the relation 
$Tm \neq \{ 0 \}$ ($mT' \neq \{ 0 \}$) holds.
If $M$ is a $T$--$T'$-bimodule which is torsion-free both as a left $T$-module and as a right $T'$-module, then $M$ is said to be \emph{torsion-free}.
It is clear that the following chain of implications hold:
\begin{equation}\label{implications}
\mbox{$M$ is unital $\Rightarrow$ $M$ is s-unital $\Rightarrow$ $M$ is torsion-free.}
\end{equation}

\section{Graded rings}\label{sec:gradedrings}

In this section, we provide the necessary background on 
group graded rings. In particular, we define the notion of 
a symmetrically graded ring, an epsilon-strongly graded ring, 
a nearly epsilon-strongly graded ring, a non-degenerately graded ring
resp. a strongly non-degenerately graded ring
(see Definitions~\ref{def:symmetricallygraded}, \ref{def:epsilongraded}, \ref{def:nearlyepsilon},
\ref{def:nondegenerate} resp. \ref{def:weaklynondegenerate})
and show a result connecting these notions (see Proposition~\ref{epsilonweakly}).

Throughout this section, $G$ denotes a group with identity element $e$
and $S$ denotes an associative ring which is \emph{$G$-graded} (or \emph{graded by $G$}).
Recall that this means that for each $g \in G$, there is an additive subgroup $S_g$ of $S$
such that $S = \oplus_{g \in G} S_g$ and for all $g,h \in G$
the inclusion $S_g S_h \subseteq S_{gh}$ holds.
If, in addition, for all $g,h \in G$ the 
equality $S_g S_h = S_{gh}$ holds, then $S$ is said to be
\emph{strongly $G$-graded} (or \emph{strongly graded by $G$}).
For an overview of the theory of group graded rings, we refer the reader to \cite{NVO2004}.

Following \cite[Definition 4.5]{CEP2016} we make the following definition.

\begin{defn}\label{def:symmetricallygraded}
The ring $S$ is said to be 
\emph{symmetrically $G$-graded}
(or \emph{symmetrically graded by $G$})
if for each $g \in G$
the relation $S_g S_{g^{-1}} S_g = S_g$ holds.
\end{defn}

We now recall from \cite{NOP2016} the following weakening of 
the notion of a unital strongly graded ring.

\begin{defn}\label{def:epsilongraded}
The ring $S$ is said to be
\emph{epsilon-strongly $G$-graded}
(or \emph{epsilon-strongly graded by $G$})
if for each $g \in G$ the $S_g S_{g^{-1}}$--$S_{g^{-1}} S_g$-bimodule $S_g$ is unital.
\end{defn}

The following characterization of epsilon-strongly graded
rings more or less follows from \cite[Proposition 7]{NOP2016}.
For the convenience of the reader we include a proof.

\begin{prop}\label{equivalent}
The following assertions are equivalent:
\begin{itemize}

\item[(i)] The ring $S$ is epsilon-strongly $G$-graded;

\item[(ii)] The ring $S$ is symmetrically $G$-graded and for each $g \in G$
the ring $S_g S_{g^{-1}}$ is unital;

\item[(iii)] For each $g \in G$ there exists $\epsilon_g \in S_g S_{g^{-1}}$
such that for all $s \in S_g$ the equalities 
$\epsilon_g s = s \epsilon_{g^{-1}} = s$ hold.

\end{itemize}
\end{prop}

\begin{proof}
(i)$\Rightarrow$(ii): 
Suppose that $S$ is epsilon-strongly $G$-graded.
Take $g \in G$. First we show that $S$ is symmetrically $G$-graded.
It is clear that $S_g S_{g^{-1}} S_g \subseteq S_g$. 
Now we show the reversed inclusion. Take $s \in S_g$.
From Definition~\ref{def:epsilongraded} it follows that there exists
$\epsilon_g \in S_g S_{g^{-1}}$ such that 
$\epsilon_g s = s$. In particular, it follows that
$s \in \epsilon_g S_g \subseteq S_g S_{g^{-1}} S_g$.
Now we show that the ring $S_g S_{g^{-1}}$ is unital.
By Definition~\ref{def:epsilongraded}, 
$S_g$ is a unital left $S_g S_{g^{-1}}$-module and
$S_{g^{-1}}$ is a unital right $S_g S_{g^{-1}}$-module.
Thus, $S_g S_{g^{-1}}$ is unital as an $S_g S_{g^{-1}}$-bimodule.

(ii)$\Rightarrow$(iii):
Suppose that $S$ is symmetrically $G$-graded and for each $g \in G$
the ring $S_g S_{g^{-1}}$ is unital. Take $g \in G$.
Let $\epsilon_g$ denote the multiplicative identity element of $S_g S_{g^{-1}}$.
Take $s \in S_g$. Using that $S$ is symmetrically $G$-graded, 
there exist $n \in \mathbb{N}$, $a_i,c_i \in S_g$ and $b_i \in S_{g^{-1}}$,
for $i \in \{1,\ldots,n\}$, such that $s = \sum_{i=1}^n a_i b_i c_i$.
Take $i \in \{ 1, \ldots , n \}$.
Since $a_i b_i \in S_g S_{g^{-1}}$ we get that
$\epsilon_g a_i b_i = a_i b_i$ and thus
$\epsilon_g s = \sum_{i=1}^n \epsilon_g a_i b_i c_i = 
\sum_{i=1}^n a_i b_i c_i = s$.
Similarly, $s \epsilon_{g^{-1}} = s$.

(iii)$\Rightarrow$(i): 
This is immediate.
\end{proof}

\begin{remark}\label{remarkid}
If $S$ is epsilon-strongly $G$-graded, then 
the element $\epsilon_g$ from Proposition~\ref{equivalent}(iii) is a multiplicative 
identity element of the $S_e$-ideal $S_g S_{g^{-1}}$ and $\epsilon_g \in Z( S_e )$.
Moreover, $\epsilon_e$ is a multiplicative identity element of both $S_e$ and $S$,
making $S$ a unital ring.
\end{remark}

\begin{prop}\label{propstronglygraded}
If $S$ is epsilon-strongly $G$-graded, then the following assertions are equivalent: 

\begin{itemize}

\item[(i)] $S$ is strongly $G$-graded;

\item[(ii)] for each $g \in G$ the equality $S_g S_{g^{-1}} = S_e$ holds;

\item[(iii)] for each $g \in G$ the equality $\epsilon_g = 1$ holds.

\end{itemize}
\end{prop}

\begin{proof}
The equivalence (i)$\Leftrightarrow$(ii) is \cite[Proposition 1.1.1(3)]{NVO2004} and the 
equivalence (ii)$\Leftrightarrow$(iii) follows from Remark \ref{remarkid}.
\end{proof}

Now we introduce the following weakening of epsilon-strongly graded rings.

\begin{defn}\label{def:nearlyepsilon}
The ring $S$ is said to be 
{\it nearly epsilon-strongly $G$-graded}
(or {\it nearly epsilon-strongly graded by $G$})
if for each $g \in G$ the $S_g S_{g^{-1}}$--$S_{g^{-1}} S_g$-bimodule $S_g$ is s-unital.
\end{defn}

Now we show an ''s-unital version'' of Proposition~\ref{equivalent}.

\begin{prop}\label{sequivalent}
The following assertions are equivalent:
\begin{itemize}

\item[(i)] The ring $S$ is nearly epsilon-strongly $G$-graded;

\item[(ii)] The ring $S$ is symmetrically $G$-graded and for each $g \in G$
the ring $S_g S_{g^{-1}}$ is s-unital;

\item[(iii)] For each $g \in G$ and each $s \in S_g$ 
there exist $\epsilon_g(s) \in S_g S_{g^{-1}}$ and $\epsilon_g'(s) \in S_{g^{-1}} S_g$
such that the equalities $\epsilon_g(s) s = s \epsilon_g'(s) = s$ hold.

\end{itemize}
If {\rm (i)}, {\rm (ii)} and {\rm (iii)} are satisfied, then for all $g \in G$, the $S_e$-ideal $S_g S_{g^{-1}}$ is generated by 
\begin{displaymath}
	\{ \epsilon_g(s) \mid s \in S_g \} \cup \{ \epsilon_{g^{-1}}'(s) \mid s \in S_g \}.
\end{displaymath}
\end{prop}

\begin{proof}
(i)$\Rightarrow$(ii): 
Suppose that $S$ is nearly epsilon-strongly $G$-graded.
The symmetrical part is shown in the same way as for Proposition~\ref{equivalent}.
Take $g\in G$.
By Definition~\ref{def:nearlyepsilon}, 
$S_g$ is an s-unital left $S_g S_{g^{-1}}$-module and
$S_{g^{-1}}$ is an s-unital right $S_g S_{g^{-1}}$-module.
Thus, the ring $S_g S_{g^{-1}}$ is s-unital.

(ii)$\Rightarrow$(iii):
Suppose that the ring $S$ is symmetrically $G$-graded and for each $g \in G$
the ring $S_g S_{g^{-1}}$ is s-unital. Take $g \in G$ and $s \in S_g$.
Using that $S$ is symmetrically $G$-graded, there exist $n \in \mathbb{N}$, 
$a_i,c_i \in S_g$ and $b_i \in S_{g^{-1}}$,
for $i \in \{1,\ldots,n\}$, such that $s = \sum_{i=1}^n a_i b_i c_i$.
Since $a_i b_i \in S_g S_{g^{-1}}$ and $b_i c_i \in S_{g^{-1}} S_g$,
for $i \in \{1,\ldots,n\}$, it follows from s-unitality and Remark~\ref{tominaga}
that there exist $\epsilon_g(s) \in S_g S_{g^{-1}}$ and $\epsilon_g'(s) \in S_{g^{-1}} S_g$
such that the equalities $\epsilon_g(s) a_i b_i = a_i b_i$ and
$b_i c_i \epsilon_g'(s) = b_i c_i$ hold, for $i \in \{1, \ldots, n\}$.
Thus, $\epsilon_g(s) s = s \epsilon_g'(s) = s$.

(iii)$\Rightarrow$(i): 
This is immediate.
\end{proof}

Following Cohen and Montgomery \cite{CM1984}, and 
Lundstr\"{o}m and \"{O}inert \cite{OL2012}, we make the following definition.

\begin{defn}\label{def:nondegenerate}
The gradation on $S$ is said to be \emph{right non-degenerate} 
(resp. \emph{left non-degenerate}) if for each $g \in G$
and each non-zero $s \in S_g$, the relation $s S_{g^{-1}} \neq \{ 0 \}$
(resp. $S_{g^{-1}} s \neq \{ 0 \})$ holds.
If the gradation on $S$ is right (resp. left) non-degenerate, 
then $S$ is said to be a \emph{right (resp. left) non-degenerately $G$-graded ring}.
\end{defn}

Now we define a
strengthening
of non-degenerate gradations.

\begin{defn}\label{def:weaklynondegenerate}
The gradation on $S$ is said to be
\emph{strongly right non-degenerate}
(resp. \emph{strongly left non-degenerate})
if for each $g \in G$ the left $S_g S_{g^{-1}}$-module
(right $S_{g^{-1}} S_g$-module) $S_g$ is torsion-free.
If the gradation on $S$ is strongly right (resp. left) non-degenerate, 
then $S$ is said to be a \emph{strongly right (resp. left) non-degenerately $G$-graded ring}.
\end{defn}

\begin{prop}\label{epsilonweakly}
If $S$ is nearly epsilon-strongly $G$-graded,
then $S$ is symmetrically $G$-graded,
strongly right non-degenerately $G$-graded
and
strongly left non-degenerately $G$-graded. 
\end{prop}

\begin{proof}
This follows immediately from \eqref{implications} and Proposition~\ref{sequivalent}(ii).
\end{proof}

\begin{defn}
An involution $(\cdot)^* : S \to S$ is said to be \emph{anti-graded} if for all
$g \in G$ the equality $(S_g)^* = S_{g^{-1}}$ holds.
\end{defn}

\begin{prop}\label{propgradedinvolution}
If $S$ is equipped with an anti-graded involution, 
then the following assertions are equivalent:
\begin{itemize}

\item[(i)] $S$ is epsilon-strongly $G$-graded;

\item[(ii)] for each $g \in G$, there exists $\epsilon_g \in S_g S_{g^{-1}}$
such that $\epsilon_g^* = \epsilon_g$ and for all $s \in S_g$ 
the equality $\epsilon_g s = s$ holds;

\item[(iii)] for each $g \in G$, there exists $\epsilon_g \in S_g S_{g^{-1}}$
such that $\epsilon_g^* = \epsilon_g$ and for all $s \in S_g$ 
the equality $s \epsilon_{g^{-1}} = s$ holds.

\end{itemize}
\end{prop}

\begin{proof}
(i)$\Rightarrow$(ii):
Suppose that $S$ is epsilon-strongly $G$-graded.
Take $g \in G$ and an element $\epsilon_g \in S_g S_{g^{-1}}$
satisfying the conditions in Proposition \ref{equivalent}(iii).
Then $\epsilon_g^* \in (S_g S_{g^{-1}})^* = (S_{g^{-1}})^* S_g^* = S_g S_{g^{-1}}$
and thus, from Remark \ref{remarkid}, it follows that
$\epsilon_g^* = \epsilon_g \epsilon_g^* = (\epsilon_g^*)^* \epsilon_g^* =
( \epsilon_g \epsilon_g^* )^* = ( \epsilon_g^* )^* = \epsilon_g$.
Thus, (ii) holds.

(ii)$\Rightarrow$(i): Suppose that (ii) holds 
for some elements $\epsilon_g \in S_g S_{g^{-1}}$, for $g \in G$. 
Take $g \in G$ and $s \in S_g$. From the fact that $s^* \in S_{g^{-1}}$
it follows that 
$s \epsilon_{g^{-1}}  = 
( \epsilon_{g^{-1}}^* s^* )^* =
( \epsilon_{g^{-1}} s^* )^* =
(s^*)^* = s$.

The proof of the equivalence (i)$\Leftrightarrow$(iii) is analogous 
and is therefore left to the reader.
\end{proof}

\begin{prop}\label{propnearlyinvolution}
Suppose that $S$ is equipped with an anti-graded involution.
\begin{itemize}

\item[(a)] If for each $g \in G$ and all $s \in S_g$,
there exists $\epsilon_g(s) \in S_g S_{g^{-1}}$
such that $\epsilon_g(s)^* = \epsilon_g(s)$ and 
the equality $\epsilon_g(s) s = s$ holds,
then $S$ is nearly epsilon-strongly $G$-graded.
In that case, for each $g \in G$, the $S_e$-ideal $S_g S_{g^{-1}}$ is generated by 
$\{ \epsilon_g(s) \mid s \in S_g \}.$

\item[(b)] If for each $g \in G$ and all $s \in S_g$, 
there exists $\epsilon_g'(s) \in S_{g^{-1}} S_g$
such that $\epsilon_g'(s)^* = \epsilon_g'(s)$ and 
the equality $s \epsilon_g'(s)= s$ holds,
then $S$ is nearly epsilon-strongly $G$-graded.
In that case, for each $g \in G$, the $S_e$-ideal $S_g S_{g^{-1}}$ is generated by 
$\{ \epsilon_{g^{-1}}'(s) \mid s \in S_g \}.$

\end{itemize}
\end{prop}

\begin{proof}
Analogous to the proofs of (ii)$\Rightarrow$(i) resp. (iii)$\Rightarrow$(i)
in Proposition \ref{propgradedinvolution}. 
\end{proof}

\section{Leavitt path algebras}\label{sec:leavittpathalgebras}

In this section, we accomodate the Leavitt path algebra 
framework that we need; in particular standard $G$-gradations
on Leavitt path algebras, for an arbitrary group $G$. 
We prove $G$-graded versions of Theorem~\ref{maintheoremfinite}
and Theorem~\ref{maintheorem} (see Theorem~\ref{finitegraphepsilonstrongly}
and Theorem~\ref{graphepsilonstrongly}).
At the end of this section we
establish criteria
which ensure that a Leavitt path algebra is Frobenius over its identity component
when the group $G$ is finite (see Theorem~\ref{maintheoremfrobenius}).

Let $R$ be an associative unital ring and let
$E = (E^0,E^1,r,s)$ be a directed graph.
Recall that 
$r$ (range) and $s$ (source) are maps $E^1 \to E^0$. The elements of $E^0$ are called \emph{vertices} 
and the elements of $E^1$ are called \emph{edges}. 
A vertex $v$ for which $s^{-1}(v)$ is empty is called a \emph{sink}.
A vertex $v$ for which $r^{-1}(v)$ is empty is called a \emph{source}.
If $s^{-1}(v)$ is a finite set for every $v \in E^0$,
then $E$ is called \emph{row-finite}.
If both $E^0$ and $E^1$ are finite sets, then we say that $E$ is \emph{finite}.
A \emph{path} $\mu$ in $E$ is a sequence of edges 
$\mu = \mu_1 \ldots \mu_n$ such that $r(\mu_i)=s(\mu_{i+1})$ 
for $i\in \{1,\ldots,n-1\}$. In such a case, $s(\mu):=s(\mu_1)$ 
is the \emph{source} of $\mu$, $r(\mu):=r(\mu_n)$ is the \emph{range} 
of $\mu$, and $l(\mu):=n$ is the \emph{length} of $\mu$.
For any vertex $v \in E^0$ we put $s(v):=v$ and $r(v):=v$.
The elements of $E^1$ are called \emph{real edges}, while for $f\in E^1$
we call $f^*$ a \emph{ghost edge}.
The set $\{f^* \mid f \in E^1\}$ will be denoted by $(E^1)^*$.
We let $r(f^*)$ denote $s(f)$, and we let $s(f^*)$ denote $r(f)$.
Following Hazrat \cite{H2013A} we make the following definition.

\begin{defn}\label{def:LPA}
The \emph{Leavitt path algebra of $E$ with coefficients in $R$}, denoted by $L_R(E)$,
is the algebra generated by the sets
$\{v \mid v\in E^0\}$, $\{f \mid f\in E^1\}$ and $\{f^* \mid f\in E^1\}$
with the coefficients in $R$,
subject to the relations:
\begin{enumerate}
	\item $uv = \delta_{u,v} v$ for all $u,v \in E^0$;
	\item $s(f)f=fr(f)=f$ and $r(f)f^*=f^*s(f)=f^*$, for all $f\in E^1$;
	\item $f^*f'=\delta_{f,f'} r(f)$, for all $f,f'\in E^1$;
	\item $\sum_{f\in E^1, s(f)=v} ff^* = v$, for every $v\in E^0$ for which $s^{-1}(v)$ is non-empty and finite.
\end{enumerate}
Here the ring $R$ commutes with the generators.
\end{defn}

There are many ways to assign a group gradation to $L_R(E)$.
Indeed, let $G$ be an arbitrary group.
Put $\degree(v)=e$, for each $v\in E^0$.
For each $f\in E^1$, choose some $g\in G$ and put
$\deg(f) = g$ and $\deg(f^*) = g^{-1}$.
By assigning degrees to the generators of the free algebra $F_R(E)=R\langle v,f,f^* \mid v\in E^0, f\in E^1 \rangle$ in this way, we are ensured
that the ideal coming from relations (1)--(4) in Definition~\ref{def:LPA}
is \emph{graded}, i.e. homogeneous.
Using this it is easy to see that the natural $G$-gradation on $F_R(E)$ carries over to a $G$-gradation on the quotient algebra
$L_R(E)$.
We will refer to such a gradation as \emph{a standard $G$-gradation on $L_R(E)$}.
If $\mu=\mu_1\ldots \mu_n$ is a path, then $\mu$ can be viewed as an element in $L_R(E)$, and we denote by $\mu^*$ the element
$\mu_n^*\cdots \mu_1^*$ of $L_R(E)$.
This gives rise to an anti-graded involution on $L_R(E)$.
For $n\geq 1$ we define $E^n$ to be the set of paths of length $n$,
and let $E^* = \cup_{n\geq 0} E^n$ be the set of all paths (and vertices).
It is clear that the map $\degree : E^0 \cup E^1 \cup (E^1)^* \to G$, described above,
can be extended to a map
$\degree : E^* \cup \{\mu^* \mid \mu \in E^* \} \to G$.
One important instance of such a gradation is coming from the case $G=\Z$, when we put
$\deg(v) = 0$, for $v \in E^0$, 
$\deg(f) = 1$ and $\deg(f^*)=-1$, for $f \in E^1$.
We will refer to this gradation as the \emph{canonical $\Z$-gradation on $L_R(E)$}.\\

{\bf Assumption:} Unless otherwise stated, throughout the rest of this section
$E$ will denote an arbitrary directed graph, $R$ will denote
an arbitrary associative unital ring, $G$ will denote an arbitrary group
and we will assume that $L_R(E)$ is equipped with an arbitrary standard $G$-gradation, $S=L_R(E)= \oplus_{g\in G} S_g$.\\

We wish to introduce a partial order related to Leavitt
path algebras. To do that we need some concepts
and a well-known result which we include for the convenience of the reader. 

\begin{defn}
Let $\leq$ be a relation on a set $X$.
Recall that $\leq$ is called a {\it preorder} if
it is reflexive and transitive.
If $\leq$ is a preorder, then it is called
a {\it partial order} if it is antisymmetric.
\end{defn}

\begin{prop}\label{quotient}
Suppose that $\leq$ is a preorder on a set $X$.
If we, for any $x,y \in X$, put $x \sim y$,
whenever $x \leq y$ and $y \leq x$, then
$\sim$ is an equivalence relation on $X$
which makes the quotient relation $\preceq$ on $X / \! \sim$,
induced by $\leq$, a partial order.
\end{prop}

\begin{proof}
This is well-known and straightforward to check.
\end{proof}

\begin{defn}\label{def:setXg}
Let $X$ denote the set of all formal expressions
of the form $\alpha \beta^*$ where $\alpha,\beta \in E^*$
and $r( \alpha ) = r ( \beta )$.
Take $g \in G$ and put 
$X_g = \{ x \in X \mid \degree(x) = g \}.$
Suppose that $\alpha,\beta,\gamma,\delta \in E^*$
satisfy $\alpha \beta^* , \gamma \delta^* \in X_g$.
We put $\alpha \beta^* \leq \gamma \delta^*$ if
$\alpha$ is an initial subpath of $\gamma$.
We put $\alpha \beta^* \sim \gamma \delta^*$
whenever $\alpha \beta^* \leq \gamma \delta^*$
and $\gamma \delta^* \leq \alpha \beta^*$, that is
if $\alpha = \gamma$.
\end{defn}

\begin{remark}\label{rem:XandXg}
(a) In the above definition of $X$, it is fully possible to let $\alpha$ or $\beta$ be a vertex.

(b) For each $g\in G$, the homogeneous component $S_g$
in the $G$-gradation of $L_R(E)$
is the $R$-linear span of $X_g$.
\end{remark}

\begin{prop}\label{proppartialorder}
Let $g \in G$ be arbitrary. The following assertions hold:
\begin{itemize}

\item[(a)] The relation $\leq$ is a preorder on $X_g$.

\item[(b)] The relation $\sim$ is an equivalence relation on $X_g$.

\item[(c)] The quotient relation $\preceq$ on $X_g / \! \sim$
induced by $\leq$ is a partial order.

\end{itemize}
\end{prop}

\begin{proof}
(a) Take $\alpha_i,\beta_i \in E^*$, 
such that $\alpha_i \beta_i^* \in X_g$, for $i \in \{1,2,3\}$.
It is clear that $\alpha_1 \beta_1^* \leq \alpha_1 \beta_1^*$,
since $\alpha_1$ is an initial subpath of itself.
Thus, $\leq$ is reflexive.
Now we show that $\leq$ is transitive.
Suppose that $\alpha_1 \beta_1^* \leq \alpha_2 \beta_2^*$ 
and $\alpha_2 \beta_2^* \leq \alpha_3 \beta_3^*$.
Then $\alpha_1$ must be an initial subpath of $\alpha_3$.
Thus, $\alpha_1 \beta_1^* \leq \alpha_3 \beta_3^*$.

(b) This is clear.

(c) This follows from Proposition~\ref{quotient}.
\end{proof}

\begin{prop}
Let $g\in G$ be arbitrary.
The map $\Ng : X_g /\! \sim \, \ni [x] \mapsto \Ng(x):= x x^* \in L_R(E)_e$ is well-defined.
\end{prop}

\begin{proof}
Suppose that $[x] = [y]$. Then $x \sim y$ which implies
that $x \leq y$ and $y \leq x$.
Then $x = \alpha \beta^*$ and $y = \alpha \gamma^*$
for some $\alpha,\beta,\gamma \in E^*$.
Thus, $\Ng(x) = x x^* = \alpha \beta^* \beta \alpha^* = \alpha \alpha^* =
\alpha \gamma^* \gamma \alpha^* = y y^* = \Ng(y)$.
\end{proof}

\begin{lemma}\label{lemmasubpath}
If $\alpha,\beta \in E^*$ are chosen so that
$\alpha^* \beta \neq 0$ in $L_R(E)$, then
$\alpha$ is an initial subpath of $\beta$
or vice versa.
\end{lemma}

\begin{proof}
Put $m=l( \alpha )$ and $n=l( \beta )$.
Suppose that $\alpha = f_1 \ldots f_m$ and
$\beta = f_1' \ldots f_n'$.

Case 1: $m \leq n$. Then
\begin{displaymath}
	0 \neq \alpha^* \beta = 
f_m^* \cdots f_1^* f_1' \cdots f_n' =
\delta_{f_1 , f_1'} \cdots \delta_{f_m , f_m'} f_{m+1}' \cdots f_n'.
\end{displaymath}
This implies that $f_i = f_i'$ for $i \in \{1,\ldots,m\}$.
Thus, $\alpha$ is an initial subpath of $\beta$.

Case 2: $m > n$. Then
\begin{displaymath}
	0 \neq \alpha^* \beta = 
f_m^* \cdots f_1^* f_1' \cdots f_n' =
\delta_{f_1 , f_1'} \cdots \delta_{f_n , f_n'} f_m^* \cdots f_{n+1}^*.
\end{displaymath}
Hence, $f_i = f_i'$ for $i \in \{1,\ldots,n\}$.
Thus, $\beta$ is an initial subpath of $\alpha$.
\end{proof}

\begin{prop}\label{propsubpath}
Let $g \in G$ and $x,y \in X_g$ be arbitrary.
The following assertions hold:
\begin{itemize}

\item[(a)] If $[x] \preceq [y]$, then $\Ng(x) y = y$.

\item[(b)] If $[x] \npreceq [y]$ and $[y] \npreceq [x]$,
then $\Ng(x) y = 0$.

\end{itemize}
\end{prop}

\begin{proof}
Put $x = \alpha \beta^*$ and $y = \gamma \delta^*$.

(a) Suppose that $[x] \preceq [y]$.
Then $\alpha$ is an initial subpath of $\gamma$.
Hence, $\gamma = \alpha \alpha'$ for some 
$\alpha' \in E^*$ such that $r(\alpha)=s(\alpha')$. Therefore, 
$\Ng(x) y = 
\alpha \beta^* \beta \alpha^* \gamma \delta^*  =
\alpha \alpha^* \alpha \alpha' \delta^* =
\alpha \alpha' \delta^* =
\gamma \delta^* = y$.

(b) From Lemma~\ref{lemmasubpath}, we get that $\Ng(x) y = 
\alpha \beta^* \beta \alpha^* \gamma \delta^*  =
\alpha \alpha^* \gamma \delta^* = \alpha 0 \delta^* = 0$.
\end{proof}

\begin{defn}
Suppose that $E$ is a finite directed graph and let $g \in G$.
Since $E$ is finite, we can, on account of Proposition~\ref{proppartialorder},
choose $n_g \in \mathbb{N}$ and $m_1,\ldots,m_{n_g} \in X_g$
such that $\{ [m_1] , \ldots , [m_{n_g}] \}$ equals the  
set of minimal elements of $X_g/ \! \sim$, with respect to $\preceq$.
We define
\begin{displaymath}
	\epsilon_g = \sum_{i=1}^{n_g} \Ng(m_i)
\end{displaymath}
in $L_R(E)$. 
Notice that we will always get $\epsilon_e = \sum_{v\in E^0} v$.
\end{defn}

The following result allows us establish Theorem~\ref{maintheoremfinite}, in particular.

\begin{theorem}\label{finitegraphepsilonstrongly}
If $E$ is a finite directed graph and $R$ is an associative
unital ring, then $L_R(E)$ is epsilon-strongly $G$-graded.
\end{theorem}

\begin{proof}
Let $g \in G$ be arbitrary and put $\epsilon_g = \sum_{i=1}^{n_g} \Ng(m_i)$ in $L_R(E)$. 
Clearly, $\epsilon_g \in S_g S_{g^{-1}}$.
Take $\gamma \delta^* \in X_g$.
There is a unique $j \in \{ 1, \ldots , n_g \}$
such that $[m_j] \preceq [ \gamma \delta^* ]$.
From Proposition~\ref{propsubpath} it follows that
$\Ng(m_j) \gamma \delta^* = \gamma \delta^*$ and
if $i \neq j$, then $\Ng(m_i) \gamma \delta^* = 0$.
Thus, $\epsilon_g \gamma \delta^* = \gamma \delta^*$.
The claim now follows from Proposition \ref{propgradedinvolution} and 
Remark~\ref{rem:XandXg}(b).
\end{proof}

\begin{remark}
Consider the directed graph $B$ from the introduction
and suppose that $L_R(B)$ is equipped with a standard $G$-gradation.
From Theorem~\ref{finitegraphepsilonstrongly} it follows that
$L_R(B)$ is epsilon-strongly $G$-graded.
\end{remark}

The following result allows us establish Theorem~\ref{maintheorem}, in particular.

\begin{theorem}\label{graphepsilonstrongly}
If $E$ is a directed graph and $R$ is an associative unital ring, 
then $L_R(E)$ is nearly epsilon-strongly $G$-graded.
\end{theorem}

\begin{proof}
Take $g \in G$ and $s \in S_g$.
There is a finite subset $I$ of $\mathbb{N}$ such that
$s = \sum_{i \in I} r_i \alpha_i \beta_i^*$
for suitable $r_i \in R$ and $\alpha_i \beta_i^* \in X_g$.
Take a subset $J$ of $I$ such that
$\{ [ \alpha_j \beta_j^* ] \}_{j \in J}$ equals the set of
minimal elements in $\{ [ \alpha_i \beta_i^* ] \}_{i \in I}$.
Put $\epsilon_g(s) = \sum_{j \in J} \Ng( \alpha_j \beta_j^* ) \in S_g S_{g^{-1}}$.
For each $i \in I$, there is a unique $j(i) \in J$ such that 
$[ \alpha_{j(i)} \beta_{j(i)}^* ] \preceq [ \alpha_i \beta_i^* ]$.
From Proposition~\ref{propsubpath} it follows that
$\Ng( \alpha_{j(i)} \beta_{j(i)}^* ) \alpha_i \beta_i^* = \alpha_i \beta_i^*$ 
and if $j \neq j(i)$, then 
$\Ng( \alpha_j \beta_j^* ) \alpha_i \beta_i^* = 0$.
Thus, $\epsilon_g(s) s = 
\sum_{i \in I} \sum_{j \in J} r_i \Ng( \alpha_j \beta_j^* ) \alpha_i \beta_i^* =
\sum_{i \in I} r_i \Ng( \alpha_{j(i)} \beta_{j(i)}^* ) \alpha_i \beta_i^* =
\sum_{i \in I} r_i \alpha_i \beta_i^* = s$.
The claim now follows from Proposition \ref{propnearlyinvolution}.
\end{proof}

\begin{exmp}\label{exmp:degreeproblem}
Consider the following finite directed graph $E$:
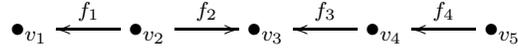
\begin{displaymath}
	\xymatrix{
	\bullet_{v_1} & \ar[l]_{f_1} \bullet_{v_2}  \ar[r]^{f_2} & \bullet_{v_3} & \ar[l]_{f_3} \bullet_{v_4} & \ar[l]_{f_4} \bullet_{v_5}
	}
\end{displaymath}
Equip $S=L_R(E)$ with its canonical $\Z$-gradation, and view it as a $\Z$-graded ring.
Notice that $v_1$ and $v_3$ are sinks. Hence, by Theorem~\ref{thm:Hazrat2}, $S$ is not strongly $\Z$-graded.
However, $S$ is epsilon-strongly $\Z$-graded, by Theorem~\ref{finitegraphepsilonstrongly}.
We shall now, for each $n\in \Z$, explicitly describe the homogeneous component $S_n$ and the element $\epsilon_n$.
A quick calculation
yields the following:
\begin{align*}
	&S_0 = \spanK_R \{ v_1, v_2, v_3, v_4, v_5, f_1f_1^*, f_2f_3^*, f_2f_2^*, f_3f_2^* \}, \\
	&S_1 = \spanK_R \{ f_1, f_2, f_3, f_4, f_4f_3f_2^* \}, \quad \quad
	S_{-1} = \spanK_R \{ f_1^*, f_2^*, f_3^*, f_4^*, f_2f_3^*f_4^* \}, \\
	&S_2 = \spanK_R \{ f_4f_3 \}, \quad \quad S_{-2} = \spanK_R \{ f_3^*f_4^* \}, \quad \text{and} \quad S_{n} = \{0\}, \text{ for } |n|>2.
\end{align*}
Following the proof of Theorem~\ref{finitegraphepsilonstrongly}
we shall now choose:
\begin{itemize}
	\item $\epsilon_0 = v_1 + v_2 + v_3 + v_4 + v_5 \in S_0$;
	\item $\epsilon_1 = v_2+v_4+v_5 = (f_2f_2^*+f_1f_1^*) + f_3f_3^* + f_4f_4^* \in S_1 S_{-1}$;
	\item $\epsilon_{-1} = f_2f_2^* + v_1 + v_3 + v_4 = f_2f_3^*f_4^* f_4f_3f_2^*
	+ f_1^*f_1 + f_3^*f_3 + f_4^* f_4 \in S_{-1} S_1$;
	\item $\epsilon_2 = v_5 = f_4f_3 f_3^*f_4^* \in S_2 S_{-2}$;
	\item $\epsilon_{-2} = v_3 = f_3^*f_4^* f_4f_3 \in S_{-2} S_2$;
	\item $\epsilon_n = 0$, if $|n|>2$.
\end{itemize}
It is not difficult to see that, for any $n\in \Z$, $\epsilon_n$ satisfies the conditions in 
Proposition~\ref{equivalent}(iii).
\end{exmp}

The next example shows that Theorem~\ref{finitegraphepsilonstrongly} can not be generalized to
arbitrary (possibly non-finite) directed graphs $E$.

\begin{exmp}\label{exmp:notepsilongraded}
Consider the graph from the introduction
$C:	\xymatrix{
	\bullet_{v_1} \ar[r]^{(\infty)} & \bullet_{v_2}. \\
	}$
For this graph $C^1=\{f_1,f_2,f_3,\ldots\}$ is infinite.
Is $L_R(C)$ epsilon-strongly $\Z$-graded?
We begin by writing out
some of the homogeneous components of the canonical $\Z$-gradation on $S=L_R(C)$.
\begin{itemize}
	\item $S_0 = \spanK_R \{ v_1, v_2, f_1 f_1^*, f_1 f_2^*, \ldots, f_i f_j^* , \ldots \}$
	\item $S_1 = \spanK_R \{f_1, f_2, f_3, \ldots \}$
	\item $S_{-1} = \spanK_R \{f_1^*, f_2^*, f_3^*, \ldots \}$
\end{itemize}
If $S_1$ had only contained a finite number of elements,
then we could have put $\epsilon_1 = \sum_{f \in C^1} ff^*$.
However, in the current situation that is not possible.
The only option left for us is to put $\epsilon_1 = v_1$.
We need to check whether the conditions in Definition~\ref{def:epsilongraded} are satisfied,
and in particular whether $v_1 \in S_1 S_{-1}$.
Seeking a contradiction, suppose that
$v_1 = \sum_{i\in F} r_i f_{p_i} f_{q_i}^*$,
where $p_i,q_i \in \N^+$,
for some finite set $F$.
For any $k \in \N^+$ we get
\begin{displaymath}
	0 \neq v_2 = f_k^* v_1 f_k = \sum_{i\in F} r_i f_k^* f_{p_i} f_{q_i}^* f_k
\end{displaymath}
which in particular yields that there is some $i\in F$ such that
$k=p_i=q_i$. Thus, $F$ is not finite, which is a contradiction.
This shows that $\epsilon_1 \notin S_1 S_{-1}$
and hence $S=L_R(C)$ is not epsilon-strongly $\Z$-graded.
\end{exmp}

\begin{remark}
By Example~\ref{exmp:degreeproblem} and
Example~\ref{exmp:notepsilongraded} we have
established the following chain of strict inclusions:
\begin{align*}
& \{ \mbox{unital strongly $\Z$-graded Leavitt path algebras} \} \\
& \subsetneq \{ \mbox{epsilon-strongly $\Z$-graded Leavitt path algebras} \} \\
& \subsetneq \{ \mbox{nearly epsilon-strongly $\Z$-graded Leavitt path algebras} \}.
\end{align*}
\end{remark}

\begin{prop}\label{nondegenerate}
If $E$ is a directed graph and $R$ is an associative unital ring,
then the Leavitt path algebra $L_R(E)$ is
symmetrically $G$-graded,
strongly right non-degenerately $G$-graded
and
strongly left non-degenerately $G$-graded.
\end{prop}

\begin{proof}
This follows immediately from Proposition~\ref{epsilonweakly} and 
Theorem~\ref{graphepsilonstrongly}.
\end{proof}

\begin{remark}
The ''symmetric'' part of Proposition~\ref{nondegenerate} can be shown 
directly without
the use of
Theorem~\ref{graphepsilonstrongly}.
Indeed, take $g\in G$.
If $S_g = \{0\}$, then clearly $S_g S_{g^{-1}} S_g = S_g$ holds.
Now, suppose that $S_g \neq \{0\}$.
It follows immediately from the $G$-gradation that $S_g S_{g^{-1}} S_g \subseteq S_g$ holds.
To show the reversed inclusion, take an arbitrary non-zero monomial $\alpha \beta^* \in S_g$.
Then $r(\alpha)=r(\beta)$ and $\beta \alpha^* \in S_{g^{-1}}$. We get
$	\alpha \beta^* = \alpha r(\beta) r(\alpha) \beta^*
	= \alpha (\beta^* \beta) (\alpha^* \alpha) \beta^*
	= (\alpha \beta^*) (\beta \alpha^*) (\alpha \beta^*) \in S_g S_{g^{-1}} S_{g}.$
This shows that $S_g \subseteq S_g S_{g^{-1}} S_{g}$.
\end{remark}

We end this section with an application to Frobenius extensions.
To this end, recall the following.

\begin{defn}
Suppose that $S/T$ is a ring extension. 
By this we mean that $T \subseteq S$ and both $S$ and $T$ 
are unital with a common multiplicative identity.
The ring extension $S/T$ is called {\it Frobenius} if there is a
finite set $J$, $x_j , y_j \in S$, for $j \in J$, and a $T$-bimodule
map $E : S \rightarrow T$ such that, for every $s \in S$, the equalities
$s = \sum_{j \in J} x_j E(y_j s) = \sum_{j \in J} E(s x_j) y_j$ hold.
In that case, $(E , x_j , y_j)$ is called a {\it Frobenius system}.
\end{defn}

In \cite[Theorem 25]{NOP2016} Nystedt, \"{O}inert and Pinedo have
shown the following result.

\begin{theorem}
Suppose that $S$ is a ring which is epsilon-strongly graded by a finite group $G$. 
If we put $T = S_e$, then $S/T$ is a Frobenius extension.
\end{theorem}

As a direct consequence of the above result,
in combination with Theorem~\ref{finitegraphepsilonstrongly},
we get the following result.

\begin{theorem}\label{maintheoremfrobenius}
Let $E$ be a finite directed graph and $R$ an associative unital ring.
If $G$ is a finite group and we equip the Leavitt path algebra $L_R(E)$ with 
a standard $G$-gradation, then the ring extension 
$L_R(E)/L_R(E)_e$ is Frobenius.
\end{theorem}

\section{Strongly $\mathbb{Z}$-graded rings}\label{sec:Zgraded}

Throughout this section $S$ denotes a, not necessarily unital, $\mathbb{Z}$-graded ring.

\begin{prop}\label{prop:stronglyZgraded}
The ring $S$ is strongly $\Z$-graded if and only if 
for every $n \in \mathbb{Z}$ the $S_0$-bimodule $S_n$ is unital, 
and the equalities $S_1 S_{-1} = S_{-1} S_1 = S_0$ hold. 
\end{prop}

\begin{proof}
The ''only if'' statement is immediate.
Now we show the ''if'' statement.
Take positive integers $m$ and $n$.
First we show by induction that $S_m = (S_1)^m$.
The base case $m=1$ is clear. Next suppose that 
$S_m = (S_1)^m$. Then we get that 
$S_{m+1} = S_0 S_{m+1} = S_1 S_{-1} S_{m+1} \subseteq
S_1 S_{-1 + m+1} = S_1 S_m = S_1 (S_1)^m = (S_1)^{m+1} \subseteq S_{m+1}$.
Next we show by induction that $S_{-n} = (S_{-1})^n$.
The base case $n=1$ is clear. Next suppose that 
$S_{-n} = (S_{-1})^n$. Then we get that 
$S_{-n-1} = S_{-n-1} S_0 = S_{-n-1} S_1 S_{-1} \subseteq 
S_{-n-1+1} S_{-1} = S_{-n} S_{-1} = (S_{-1})^n S_{-1} = (S_{-1})^{n+1} \subseteq S_{-n-1}$.

Case 1: $S_m S_n = (S_1)^m (S_1)^n = (S_1)^{m+n} = S_{m+n}$.

Case 2: $S_{-m} S_{-n} = (S_{-1})^m (S_{-1})^n = (S_{-1})^{m+n} = S_{-m-n}$.

Case 3: Now we show that $S_m S_{-n} = S_{m-n}$.
We get that 
$S_m S_{-n} = (S_1)^m (S_{-1})^n$.
By repeated application of the equality $S_1 S_{-1} = S_0$,
we get that $(S_1)^m (S_{-1})^n = (S_1)^{m-n} = S_{m-n}$,
if $m \geq n$, or $(S_1)^m (S_{-1})^n = (S_{-1})^{n-m} = S_{n-m}$, otherwise.

Case 4: $S_{-m} S_n = S_{n-m}$. This is shown in a similar fashion
to Case 3, using the equality $S_{-1} S_1 = S_0$, and is therefore left to the reader.
\end{proof}

\begin{cor}\label{cor:Zstrong}
If for every $n \in \mathbb{Z}$, the $S_0$-bimodule $S_n$ is s-unital,
then $S$ is strongly $\Z$-graded if and only if the equalities $S_1 S_{-1} = S_{-1} S_1 = S_0$ hold.
\end{cor}

\begin{cor}\label{cor:symZstrong}
If $S$ is symmetrically $\mathbb{Z}$-graded, then $S$ is strongly $\Z$-graded
if and only if the equalities $S_1 S_{-1} = S_{-1} S_1 = S_0$ hold.
\end{cor}

\begin{cor}\label{cor:epsilonZstrong}
If $S$ is epsilon-strongly $\mathbb{Z}$-graded, then $S$ is strongly 
$\Z$-graded if and only if $\epsilon_1 = \epsilon_{-1} = 1$.
\end{cor}

\section{Strongly $\mathbb{Z}$-graded Leavitt path algebras}\label{sec:ZgradedLPA}

Throughout this section, $E$ denotes a directed graph,
$R$ denotes an associative unital ring and we equip
the associated Leavitt path algebra $S=L_R(E)$ with its canonical $\mathbb{Z}$-gradation.
Recall that if $E$ is finite, then $L_R(E)$ is epsilon-strongly $\Z$-graded by 
Theorem~\ref{finitegraphepsilonstrongly}.
In that case, we will use the notation established in Section~\ref{sec:leavittpathalgebras} without further mention.

\begin{lemma}\label{lem:epsilonrelations}
If $E$ is a finite directed graph, then
$\epsilon_1 = \sum_{v \in E^0 \setminus \{ {\rm sinks} \} } v$.
\end{lemma}

\begin{proof}
From the definition of the relation $\preceq$ it follows that
$\{ [ \alpha ] \mid \alpha \in E^1 \}$ is a set of minimal 
elements of $X_1 / \! \sim$. Thus
\begin{displaymath}
	\epsilon_1 = 
\sum_{\alpha \in E^1} \NOne( \alpha ) =
\sum_{\alpha \in E^1} \alpha \alpha^* = 
\sum_{v \in E^0 \setminus \{ {\rm sinks} \} } \sum_{ \stackrel{ \alpha \in E^1}{ \ s(\alpha)=v} } \alpha \alpha^* =
\sum_{v \in E^0 \setminus \{ {\rm sinks} \} } v.
\end{displaymath}
\end{proof}

\begin{lemma}\label{lem:directsum}
If $E$ is a finite directed graph which has no sink,
then for every $v \in E^0$ there exist $n(v) \in \mathbb{N}$
and $m_{1,v},\ldots,m_{n(v),v} \in X_{-1}$ such that no pair of
elements from $[m_{1,v}],\ldots,[m_{n(v),v}]$
has a common lesser element in $X_{-1} / \! \sim$, with respect to $\preceq$,
and $v = \sum_{i=1}^{n(v)} \mathcal{N}_{-1} ( m_{i,v} ) \in S_{-1} S_1$.
\end{lemma}

\begin{proof}
Take $v \in E^0$. Since $E$ is finite and has no sink, it is clear that we can make repeated use
of property (4) in Definition~\ref{def:LPA} to find $n(v) \in \mathbb{N}$ and $\alpha_{i,v} \in E^*$,
for $i \in \{1,\ldots,n(v)\}$, such that $r(\alpha_{i,v})$ is a vertex 
belonging to a cycle $C_i$ in $E$,
and $v = \sum_{i=1}^{n(v)} \alpha_{i,v} \alpha_{i,v}^*$. For each $i \in \{ 1,\ldots,n(v) \}$,
choose a path $\beta_i$ in $C_i$ such that $r(\beta_i) = r( \alpha_{i,v} )$ and 
$\alpha_{i,v} \beta_i^* \in X_{-1}$. Then $\beta_i \alpha_{i,v}^* \in X_1$.
For each $i \in \{ 1,\ldots,n(v) \}$, put $m_i = \alpha_{i,v} \beta_i^*$. Then 
\begin{align*}
v &= \sum_{i=1}^{n(v)} \alpha_{i,v} \alpha_{i,v}^* 
  = \sum_{i=1}^{n(v)} \alpha_{i,v} r( \alpha_{i,v} ) \alpha_{i,v}^* 
   = \sum_{i=1}^{n(v)} \alpha_{i,v} r( \beta_i ) \alpha_{i,v}^* \\
  &= \sum_{i=1}^{n(v)} \alpha_{i,v} \beta_i^* \beta_i \alpha_{i,v}^* 
   = \sum_{i=1}^{n(v)} m_{i,v} m_{i,v}^* 
   = \sum_{i=1}^{n(v)} \mathcal{N}_{-1}(m_{i,v}) \in S_{-1} S_1.
\end{align*}
Notice that the paths $\alpha_{1,v}, \ldots, \alpha_{i,n(v)}$ are all distinct.
Thus, no pair of elements from
\linebreak
$[m_{1,v}],\ldots,[m_{n(v),v}]$ has a common lesser element in $X_{-1} / \! \sim$.
\end{proof}

We now prove Theorem~\ref{thm:Hazrat2}
using an approach different from Hazrat's \cite{H2013A}.

\begin{prop}\label{prop:Hazrat}
If $E$ is a finite directed graph, then $L_R(E)$ is strongly $\Z$-graded if and only if $E$
has no sink.
\end{prop}

\begin{proof}
The ''only if'' statement follows from Lemma \ref{lem:epsilonrelations}, and also from Lemma~\ref{lem:NecCondLPAstrong}.
Now we show the ''if'' statement.
We claim that for all $v \in E^0$ the equality $\epsilon_{-1} v = v$ holds.
If we assume that the claim holds, then we get that
$\epsilon_{-1} = \epsilon_{-1} 1 = \sum_{v \in E^0} \epsilon_{-1} v = 
\sum_{v \in E^0} v = 1$.
From Lemma \ref{lem:epsilonrelations} we get that $\epsilon_1 = 1$.
Thus, by Corollary \ref{cor:epsilonZstrong}, we get that $L_R(E)$ is strongly $\Z$-graded.

Now we show the claim.
Take $v\in E^0$. By Lemma~\ref{lem:directsum} we may write
$v = \sum_{i=1}^{n(v)} \mathcal{N}_{-1} ( m_{i,v} )$
for suitable $n(v) \in \N$ and
$m_{1,v},\ldots,m_{n(v),v} \in X_{-1} \subseteq S_{-1}$.
Thus,
$\epsilon_{-1} v = \epsilon_{-1} \sum_{i=1}^{n(v)} \mathcal{N}_{-1} ( m_{i,v} )
= \sum_{i=1}^{n(v)} \epsilon_{-1} m_{i,v} m_{i,v}^*
= \sum_{i=1}^{n(v)} m_{i,v} m_{i,v}^* = v$.
\end{proof}

\begin{lemma}\label{lem:NecCondLPAstrong}
Let $E$ be an arbitrary directed graph.
The following assertions hold:

\begin{itemize}
	\item[(a)] If $L_R(E)$ is strongly $\Z$-graded, then $E$ has no sink.
	\item[(b)] If $E$ is row-finite and has no sink, then $S_1 S_{-1} = S_0$.
	\item[(c)] If $E$ has no source, then $S_{-1} S_1 = S_0$.
\end{itemize}
\end{lemma}

\begin{proof}
(a)
Suppose that $S=L_R(E)$ is strongly $\Z$-graded.
Seeking a contradiction, suppose that there is a sink $v$ in $E$.
Then $v \in S_0 = S_1 S_{-1}$.
Using that $v$ is a sink, we get that $v = v^2 \in v S_1 S_{-1} = \{0\}$.
This is a contradiction.

(b)
It suffices to show that $E^0 \subseteq S_1 S_{-1}$.
Take $v\in E^0$.
Then $v = \sum_{s(f)=v} ff^* \in S_1 S_{-1}$.

(c)
It suffices to show that $E^0 \subseteq S_{-1} S_1$.
Take $v\in E^0$.
Choose some $f\in E^1$ such that $v=r(f)=f^*f \in S_{-1} S_1$.
\end{proof}

The following results generalizes \cite[Theorem 1.6.15]{HazratBook}.

\begin{theorem}\label{thm:rowfinite}
Let $E$ be a row-finite directed graph which has no source.
Then $L_R(E)$ is strongly $\Z$-graded
if and only if $E$ has no sink.
\end{theorem}

\begin{proof}
The ''only if'' statement follows from Lemma~\ref{lem:NecCondLPAstrong}(a).
The ''if'' statement follows from
Lemma~\ref{lem:NecCondLPAstrong}(b)-(c),
Proposition~\ref{nondegenerate}
and Corollary~\ref{cor:symZstrong}.
\end{proof}

\end{document}